\documentclass[a4paper]{amsart}

\usepackage{microtype}
\usepackage{colortbl}

\newtheorem{theorem}{Theorem}
\newtheorem{lemma}{Lemma}

\theoremstyle{definition}
\newtheorem{definition}[theorem]{Definition}

\begin{document}

\title{The maximal order of Stern's diatomic sequence}


\author{Michael Coons}
\address{School of Mathematical and Physical Sciences\\
University of Newcastle\\
Australia}
\email{Michael.Coons@newcastle.edu.au}

\author{Jason Tyler}
\address{School of Mathematical and Physical Sciences\\
University of Newcastle\\
Australia}
\email{Jason.Tyler@newcastle.edu.au}
\date{\today}
\thanks{The research of M.~Coons was supported by ARC grant DE140100223.}

\maketitle

\begin{abstract} We answer a question of Calkin and Wilf concerning the maximal order of Stern's diatomic sequence. Specifically, we prove that $$\limsup_{n\to\infty}\frac{a(n)}{n^{\log_2 \varphi}}=\frac{\varphi^{\log_2 3}}{\sqrt{5}},$$ where $\varphi=(\sqrt{5}+1)/2$ is the golden ratio. This improves on previous results given by Berlekamp, Conway, and Guy, who showed that the limit value was bounded above by 1.25, and by Calkin and Wilf, who showed that the exact value was in the interval $[({\varphi}/{\sqrt{5}})({3}/{2})^{\log_2\varphi},(\varphi+1)/{\sqrt{5}}].$
\end{abstract}

\section{Introduction}

{\em Stern's Diatomic sequence} (commonly called {\em Stern's sequence}), $\{a(n)\}_{n\geq 0}$, is given by $a(0)=0$, $a(1)=1$, and when $n\geq 1$, by $$a(2n)=a(n)\quad\mbox{and}\quad a(2n+1)=a(n)+a(n+1).$$ In a recent survey article, Northshield \cite{N2010} restated a question of Calkin and Wilf \cite{CW1998}, which asks for the exact value of $\limsup_{n\to\infty}{a(n)}/{n^{\log_2 \varphi}}$, where $\varphi=(\sqrt{5}+1)/2$ is the golden ratio and $\log_2 n$ denotes the base-$2$ logarithm of $n$; that is, they asked one to {\em determine the exact maximal order of the Stern sequence}. This question goes back at least to the 1982 book of Berlekamp, Conway, and Guy \cite[page 115]{BCG1982} who showed that $a(n-1)$ is the number of nim-sums corresponding to a given ordinary sum $n$, and gave an upper bound of 1.25 for the limit in question. Calkin and Wilf \cite{CW1998} improved on the bounds; they showed that $$0.958854\cdots=\frac{\varphi}{\sqrt{5}}\left(\frac{3}{2}\right)^{\log_2\varphi}\leq \limsup_{n\to\infty}\frac{a(n)}{n^{\log_2 \varphi}}\leq\frac{\varphi+1}{\sqrt{5}}=1.170820\cdots.$$ We answer this question by proving the following theorem.

\begin{theorem}\label{main} Let $\{a(n)\}_{n\geq 0}$ denote the Stern sequence. Then $$\limsup_{n\to\infty}\frac{a(n)}{n^{\log_2 \varphi}}=\frac{\varphi}{\sqrt{5}}\left(\frac{3}{2}\right)^{\log_2\varphi}=\frac{\varphi^{\log_2 3}}{\sqrt{5}}=0.9588541900\cdots.$$
\end{theorem}

\section{Preliminaries}

It is well-known that the maximum value of $a(m)$ in the interval $2^{n-2}\leq m\leq 2^{n-1}$ is the $n$th Fibonacci number $F_n$ and that this maximum first occurs at $$m_n:=\frac{1}{3}(2^n-(-1)^n);$$ see Lehmer \cite{L1929} and Lind \cite{L1969}.

For our proof, we will use the points $(m_n,a(m_n))$ to produce a continuous function $h(x)$ which is an upper bound for the Stern sequence, and which is asympotitically a lower bound for the function $(\varphi^{\log_2 3}/{\sqrt{5}})x^{\log_2 \varphi}$. We will then use these functions combined with some properties of limits to prove Theorem \ref{main}.

Before giving the proof of Theorem \ref{main}, we give a formal definition of $h(x)$ and provide some useful lemmas concerning $h(x)$ and its relationships to both $a(n)$ and $x^{\log_2 \varphi}$.

\begin{definition} Let $h:\mathbb{R}_{\geq 0}\to\mathbb{R}_{\geq 0}$ denote the piecewise linear function connecting the set of points $\{(0,0)\}\cup\{(m_n,a(m_n)):n\geq 2\}.$
\end{definition}

By definition, $h(x)$ is continuous in $\mathbb{R}_{\geq 0}$ and differentiable in the intervals $(m_n,m_{n+1}).$ Using point-slope form, for $x\in[m_n,m_{n+1}]$, we have \begin{align}\nonumber h(x) &=\frac{a(m_{n+1})-a(m_n)}{m_{n+1}-m_n}(x-m_n)+a(m_n)\\
\label{hfib}&=3\cdot\frac{F_{n-1}}{2^{n}+2(-1)^{n}}x+F_n-F_{n-1}\cdot\frac{2^{n}-(-1)^{n}}{2^{n}+2(-1)^{n}}\\
\label{hasym}&=\frac{1}{\sqrt{5}}\left[\frac{3}{2}\left(\frac{\varphi}{2}\right)^{n-1}x+ \varphi^n\left(1-\frac{(-1)^{n-1}}{\varphi^{2(n-1)}}-\frac{1}{\varphi}\right)\right]\cdot\left(1+O(2^{-n})\right).
\end{align} Here we have used Binet's formula that $F_n=\frac{\varphi^n-(-\varphi)^{-n}}{\sqrt{5}},$ where $\varphi=\frac{1+\sqrt{5}}{2}$ is the golden ratio.

\begin{lemma}\label{hineq} For all $x\geq 5$, we have $$h(4x+1)>h(2x+1)+h(x),$$ and $$h(4x-1)>h(2x-1)+h(x).$$
\end{lemma}

\begin{proof} Consider first the numbers $x,2x+1,$ and $4x+1$. Suppose that $x\in[m_n,m_{n+1})$ for $n\geq 4$, so that the interval is of length at least $5$. Then since $m_{n+1}=2m_n+(-1)^n$, we have $2x+1\in[m_{n+1},m_{n+2}]$ and $4x+1\in[m_{n+2},m_{n+3}]$, so that $x, 2x+1$ and $4x+1$ can be taken from different (yet consecutive) subintervals $[m_i,m_{i+1}]$. 

To make this completely clear, we consider minimal and maximal values for $x$, specifically $x=m_n$ or $x=m_{n+1}-1$. As stated in the previous paragraph, let $n\geq 4$, so that the subintervals of concern are at least of length $5$. For minimal $x\in[m_n,m_{n+1})$, we have $x=m_n$, thus $$2x+1=2m_n+1=2m_n+(-1)^n-(-1)^n+1=m_{n+1}-(-1)^n+1\in[m_{n+1},m_{n+2}],$$ and so \begin{multline*} 4x+1=2(2x+1)-1=2(m_{n+1}-(-1)^n+1)-1\\
=2(m_{n+1}+(-1)^{n+1}+1)-1=2m_{n+1}+2(-1)^{n+1}+2-1\\
=m_{n+2}+(-1)^{n+1}+1\in[m_{n+2},m_{n+3}].\end{multline*} For maximal $x\in[m_n,m_{n+1})$, we have $x=m_{n+1}-1,$ thus \begin{multline*} 2x+1=2(m_{n+1}-1)+1=2m_{n+1}+(-1)^{n+1}-(-1)^{n+1}-1\\=m_{n+2}+(-1)^n-1\in[m_{n+1},m_{n+2}],\end{multline*} and so \begin{multline*} 4x+1=2(2x+1)-1=2(m_{n+2}+(-1)^n-1)-1=2m_{n+2}+2(-1)^n-3\\=2m_{n+2}+(-1)^{n+2}+(-1)^{n+2}-3=m_{n+3}+(-1)^{n+2}-3\in[m_{n+2},m_{n+3}].\end{multline*} Since all other values of $x\in[m_n,m_{n+1})$ are strictly between the minimal and maximal values, we have shown that for $n\geq 4$, if $x\in[m_n,m_{n+1}),$ then $2x+1\in[m_{n+1},m_{n+2}]$ and $4x+1\in[m_{n+2},m_{n+3}]$.

For $x\in[m_n,m_{n+1})$, using \eqref{hfib} and the fact that $x, 2x+1$ and $4x+1$ can be taken from different (yet consecutive) subintervals $[m_i,m_{i+1}]$, we have that \begin{align*} h&(4x+1)-h(2x+1)-h(x)\\
&=\left\{12\cdot\frac{F_{n+1}}{2^{n+2}+2(-1)^{n+2}}-6\cdot\frac{F_{n}}{2^{n+1}+2(-1)^{n+1}}-3\cdot\frac{F_{n-1}}{2^{n}+2(-1)^{n}}\right\}x\\
&\quad+\left\{3\cdot\frac{F_{n+1}}{2^{n+2}+2(-1)^{n+2}}-3\cdot\frac{F_{n}}{2^{n+1}+2(-1)^{n+1}}-0\cdot\frac{F_{n-1}}{2^{n}+2(-1)^{n}}\right\}\\
&\qquad-F_{n+1}\cdot\frac{2^{n+2}-(-1)^{n+2}}{2^{n+2}+2(-1)^{n+2}}+F_{n}\cdot\frac{2^{n+1}-(-1)^{n+1}}{2^{n+1}+2(-1)^{n+1}}+F_{n-1}\cdot\frac{2^{n}-(-1)^{n}}{2^{n}+2(-1)^{n}}\\
&=\mathfrak{S}_1\cdot x+\mathfrak{S}_2+\mathfrak{S}_3,
\end{align*} where $\mathfrak{S}_1,$ $\mathfrak{S}_2$ and $\mathfrak{S}_3$ represent the three-term sums from the three previous lines, respectively. We have 
\begin{align*}
\left|\mathfrak{S}_1\cdot x\right| &\leq\left|12\cdot\frac{F_{n+1}}{2^{n+2}+2(-1)^{n+2}}-6\cdot\frac{F_{n}}{2^{n+1}+2(-1)^{n+1}}-3\cdot\frac{F_{n-1}}{2^{n}+2(-1)^{n}}\right|\cdot m_{n+1}\\
&=2\cdot\left|F_{n+1}\cdot(1+O(2^{-n}))-F_{n}\cdot(1+O(2^{-n}))-F_{n-1}\cdot(1+O(2^{-n}))\right|\\
&=O\left(\frac{\varphi^n}{2^n}\right),
\end{align*} where for the last equality we have used both the Fibonacci recursion and the fact that $F_n=O(\varphi^n)$. Using $F_n=O(\varphi^n)$ again, we immediately gain $$\left|\mathfrak{S}_2\right|=O\left(\frac{\varphi^n}{2^n}\right),$$ and similarly $$\left|\mathfrak{S}_3\right|=\left|F_{n+1}(1+O(2^{-n}))-F_{n}(1+O(2^{-n}))-F_{n-1}(1+O(2^{-n}))\right|=O\left(\frac{\varphi^n}{2^n}\right).$$ Thus $$|h(4x+1)-h(2x+1)-h(x)|\leq |\mathfrak{S}_1\cdot x|+|\mathfrak{S}_2|+|\mathfrak{S}_3|=O\left(\frac{\varphi^n}{2^n}\right).$$ Noting that $2>\varphi$, gives then that \begin{equation*}\lim_{x\to\infty}\big\{h(4x+1)-h(2x+1)-h(x)\big\}=0.\end{equation*}

In addition, this limit is strictly decreasing to zero. To see this, we suppose that $x\in[m_{n},m_{n+1})$ and use the above-established fact that $x,2x+1,$ and $4x+1$ are contained in consecutive subintervals. By \eqref{hfib} we have \begin{align*} \frac{d}{dx}\big\{h(4x+1)&-h(2x+1)-h(x)\big\}\\ 
&= 3\left[\frac{F_{n+1}}{2^{n+2}-2(-1)^{n+2}}-\frac{F_{n}}{2^{n+1}-2(-1)^{n+1}}-\frac{F_{n-1}}{2^{n}-2(-1)^{n}}\right]\\
&<3\left[\frac{F_{n+1}}{2^{n+2}-2}-\frac{F_{n}}{2^{n+1}+2}-\frac{F_{n-1}}{2^{n}+2}\right]\\
&=3\left[\frac{F_{n}+F_{n-1}}{2^{n+2}-2}-\frac{F_{n}}{2^{n+1}+2}-\frac{F_{n-1}}{2^{n}+2}\right]\\
&=3\left[\left(\frac{F_{n}}{2^{n+2}-2}-\frac{F_{n}}{2^{n+1}+2}\right)+\left(\frac{F_{n-1}}{2^{n+2}-2}-\frac{F_{n-1}}{2^{n}+2}\right)\right]\\
&<0.
\end{align*} Thus the function $h(4x+1)-h(2x+1)-h(x)$ is strictly decreasing to zero over the intervals $[m_{i},m_{i+1})$ for $i\geq 4$ and so on these intervals, we have \begin{equation*} h(x)+h(2x+1)<h(4x+1).\end{equation*} The result follows as these intervals partition the real numbers $x\geq 5$.

For the second part of the lemma, by \eqref{hasym} $$\lim_{\substack{x\to\infty\\ x\neq m_n}}\frac{d}{dx}h(x)=\lim_{n\to\infty}\frac{3}{2\sqrt{5}}\left(\frac{\varphi}{2}\right)^{n-1}\left(1+O(2^{-n})\right)= 0,$$ and since $h(x)$ is continuous for all $x\in\mathbb{R}_{\geq 0}$, we have that for any fixed number $y$, $$\lim_{x\to\infty}\big\{h(x+y)-h(x)\big\}=0.$$ Thus $$\lim_{x\to\infty}\big\{h(4x+1)-h(2x+1)-h(x)-[h(4x-1)-h(2x-1)-h(x)]\big\}=0,$$ and so \begin{equation*}\lim_{x\to\infty}\big\{h(4x-1)-h(2x-1)-h(x)\big\}=0.\end{equation*} 

If $x,2x-1,$ and $4x-1$ are in three distinct consecutive subintervals $[m_i,m_{i+1}]$, then the desired inequality follows repeating the above argument mutatis mutandis. 

If the numbers $x,2x-1,$ and $4x-1$ are not in three distinct subintervals, then it must be the case that $x,2x-1\in[m_n,m_{n+1}]$ and $4x-1\in[m_{n+1},m_{n+2}]$ for some $n$ as the numbers $2x-1$ and $4x-1$ can be always taken in distinct consecutive subintervals since $2(2x-1)+1=4x-1$. In this case, by \eqref{hfib} we have \begin{align*} \frac{d}{dx}\big\{h(4x-1)&-h(2x-1)-h(x)\big\}\\ 
&= 3\left[\frac{F_{n+1}}{2^{n+2}-2(-1)^{n+2}}-\frac{2F_{n}}{2^{n+1}-2(-1)^{n+1}}\right]\\
&<3\left[\frac{F_{n+1}}{2^{n+2}-2}-\frac{2F_{n}}{2^{n+1}+2}\right]\\
&=3\left[\frac{F_{n}+F_{n-1}}{2^{n+2}-2}-\frac{2F_{n}}{2^{n+1}+2}\right]\\
&=3\left[\left(\frac{F_{n}}{2^{n+2}-2}-\frac{F_{n}}{2^{n+1}+2}\right)+\left(\frac{F_{n-1}}{2^{n+2}-2}-\frac{F_{n}}{2^{n+1}+2}\right)\right]\\
&<3\left[\left(\frac{F_{n}}{2^{n+2}-2}-\frac{F_{n}}{2^{n+1}+2}\right)+\left(\frac{F_{n-1}}{2^{n+2}-2}-\frac{F_{n-1}}{2^{n+1}+2}\right)\right]\\
&<0.
\end{align*} Thus, as in the previous case, the function $h(4x-1)-h(2x-1)-h(x)$ is strictly decreasing to zero over the intervals $[m_{i},m_{i+1})$ for $i\geq 4$, and so on these intervals, we have \begin{equation*}\label{hgreat} h(x)+h(2x-1)<h(4x-1).\end{equation*} This completes the proof of the lemma.
\end{proof}

\begin{lemma}\label{ahlimsup} For all $m$, we have $a(m)\leq h(m).$ Moreover, $$\limsup_{m\to\infty}\frac{a(m)}{h(m)}=1.$$
\end{lemma}

\begin{proof} If $m=1,\ldots,22$, then we have $a(m)\leq h(m)$. See Table \ref{ah22} for the values of $a(m)$ and $h(m)$ to verify this.

\begin{table}[htdp]
\caption{The first 22 values for $a(m)$ and $h(m)$.}
\begin{center}
\begin{tabular}{|c|c|c|c|c|c|c|c|c|c|c|c|c|c|} \hline
$m$ & 1 & 2 & 3 & 4 & 5 & 6 & 7 & 8 & 9 & 10 & 11 & 12 & 13\\ \hline
$a(m)$ &  1 & 1& 2& 1& 3& 2& 3& 1& 4& 3& 5 & 2 & 5\\ 
$h(m)$ & 1 & 3/2 &2  & 5/2 &3  & 10/3 & 11/3 & 4 & 13/3 & 14/3 &5 & 53/10 & 28/5\\ \hline
\end{tabular}

\vspace{.2cm}
\begin{tabular}{|c|c|c|c|c|c|c|c|c|c|c|c|} \hline
$m$ & 14 & 15 & 16 & 17 & 18 & 19 & 20 & 21 & 22\\ \hline
$a(m)$ & 3& 4& 1& 5& 4& 7& 3& 8& 5\\
$h(m)$ & 59/10 & 31/5 & 13/2 & 34/5 & 71/10 & 37/5 & 77/10 &8  & 181/22\\ \hline
\end{tabular}
\end{center}
\label{ah22}
\end{table}%

Now suppose that $m\geq 23$ and that the assertion holds for all $k<m$. We consider the three cases, of $m$ even or odd modulo $4$, separately. 

If $m=2k$ for some $k$, then $$a(m)=a(2k)=a(k)\leq h(k)\leq h(2k)=h(m),$$ where we have used here that $h(x)$ is a monotone increasing function. 


If $m=4k+1$, then $m=2(2k)+1$, so that using the recursion of the Stern sequence combined with Lemma \ref{hineq} setting $x=k=\frac{m-1}{4}>\frac{22-1}{4}>5$, we have \begin{multline*} a(m)=a(2(2k)+1)=a(2k)+a(2k+1)\\=a(k)+a(2k+1) \leq h(k)+h(2k+1)<h(4k+1)=h(m).\end{multline*}

If $m=4k+3$, then $m=2(2k+1)+1$, so that using the same properties used in the previous sentence, again combined with Lemma \ref{hineq} setting $x=k+1=\frac{m+1}{4}>\frac{22+1}{5}>5$, we have \begin{multline*}a(m)=a(2(2k+1)+1)=a(2k+1)+a(2k+2) =a(2k+1)+a(k+1)\\ \leq h(k+1)+h(2k+1)=h(k+1)+h(2(k+1)-1)\\ <h(4(k+1)-1)=h(4k+3)=h(m).\end{multline*} This proves that for all $m$, $a(m)\leq h(m)$. 

For the limit result, note that \begin{equation*}1=\lim_{n\to\infty} \frac{a(m_n)}{h(m_n)}\leq \limsup_{m\to\infty}\frac{a(m)}{h(m)}\leq \limsup_{m\to\infty}\frac{h(m)}{h(m)}=1.\qedhere\end{equation*}
\end{proof}

\begin{lemma}\label{hfeps} Let $\varepsilon>0$ be given. Then for large enough $x$, we have $$\sqrt{5}\cdot h(x)\leq \varphi^{\log_23}x^{\log_2 \varphi}+\varepsilon.$$
\end{lemma}

\begin{proof} Note that at the points $(m_n,a(m_n))$, we have $h(m_n)=a(m_n)=F_n$ by construction. Thus, noting that $$\varphi^{\log_2\left(1-\frac{(-1)^n}{2^n}\right)}=1+O(2^{-n}),$$ we have \begin{align*} \sqrt{5}\cdot h(m_n)- \varphi^{\log_23}\varphi^{\log_2m_n}&=\sqrt{5}\cdot F_n- \varphi^{\log_23}\varphi^{\log_2m_n}\\ &=\varphi^n-\frac{(-1)^n}{\varphi^n}-\varphi^n\varphi^{\log_2\left(1-\frac{(-1)^n}{2^n}\right)}=O\left(\left(\frac{\varphi}{2}\right)^n\right).\end{align*} This shows that the result holds for $x=m_n$ for large enough $n$. 

Using this, let $\varepsilon>0$ be given and let $N$ be large enough so that we have \begin{equation}\label{hf1} \sqrt{5}\cdot h(m_n)-\varphi^{\log_23}\varphi^{\log_2m_n}<\varepsilon\end{equation} for all $m_n>N$ and, towards a contradiction, suppose that there is an $x_n\in(m_n,m_{n+1})$ such that \begin{equation}\label{hf2} \sqrt{5}\cdot h(x_n)-\varphi^{\log_23}\varphi^{\log_2x_n}\geq \varepsilon.\end{equation} Then since $\sqrt{5}\cdot h(x)-\varphi^{\log_23}x^{\log_2 \varphi}$ is differentiable for $x\in(m_n,m_{n+1})$, by \eqref{hf1} and \eqref{hf2}, the function attains a maximum value at some $x\in(m_n,m_{n+1})$. Thus there is an $x\in(m_n,m_{n+1})$ such that the second derivative of $\sqrt{5}\cdot h(x)-\varphi^{\log_23}x^{\log_2 \varphi}$ is negative. But \begin{align*}\frac{d^2}{dx^2}\left\{\sqrt{5}\cdot h(x)-\varphi^{\log_23}x^{\log_2 \varphi}\right\}&=\frac{d^2}{dx^2}\left\{-\varphi^{\log_23}x^{\log_2\varphi}\right\}\\ &=-\varphi^{\log_23}\log_2\varphi\cdot(\log_2\varphi-1)x^{\log_2\varphi-2},\end{align*} which is positive for all $x\in(m_n,m_{n+1})$ as $$-\varphi^{\log_23}\log_2\varphi\cdot(\log_2\varphi-1)>0.$$ Thus we arrive at the contradiction, which proves the lemma.
\end{proof}

\section{The maximal growth of Stern's diatomic sequence}

We are now in a position to prove our main result. 

\begin{proof}[Proof of Theorem \ref{main}] By Lemma \ref{hfeps} we have for large enough $x$ that \begin{equation*} h(x)\leq \frac{\varphi^{\log_23}}{\sqrt{5}}x^{\log_2 \varphi}+\varepsilon,\end{equation*} and by Lemma \ref{ahlimsup} we have $$\limsup_{n\to\infty} \frac{a(n)}{h(n)}=1.$$ Thus we have for large enough $n$ that $$\frac{a(n)}{\frac{\varphi^{\log_23}}{\sqrt{5}}n^{\log_2 \varphi}+\varepsilon}\leq\frac{a(n)}{h(n)},$$ and so we have \begin{equation}\label{star}\limsup_{n\to\infty}\frac{a(n)}{\frac{\varphi^{\log_23}}{\sqrt{5}}n^{\log_2 \varphi}+\varepsilon}\leq\limsup_{n\to\infty}\frac{a(n)}{h(n)}\leq 1.\end{equation}

Inequality \eqref{star} is true for every positive $\varepsilon$, so that \begin{equation}\label{star2} \limsup_{n\to\infty}\frac{a(n)}{\frac{\varphi^{\log_23}}{\sqrt{5}}n^{\log_2 \varphi}}\leq 1.\end{equation} 

But we also have that $$\frac{a(m_n)}{\frac{\varphi^{\log_23}}{\sqrt{5}}m_n^{\log_2\varphi}}=\frac{F_n}{\frac{\varphi^{\log_23}}{\sqrt{5}}\varphi^{\log_2m_n}}=\frac{\varphi^n-\frac{(-1)^n}{\varphi^n}}{\varphi^n\varphi^{\log_2\left(1-\frac{(-1)^n}{2^n}\right)}}=1+O(2^{-n}).$$ Thus by \eqref{star2} we have $$1=\lim_{n\to\infty}\frac{a(m_n)}{\frac{\varphi^{\log_23}}{\sqrt{5}}m_n^{\log_2\varphi}}\leq \limsup_{n\to\infty}\frac{a(n)}{\frac{\varphi^{\log_23}}{\sqrt{5}}n^{\log_2 \varphi}}\leq 1,$$ which proves the theorem.
\end{proof}

\noindent{\bf Acknowledgements.} We thank the anonymous referee for a very careful reading.


\providecommand{\bysame}{\leavevmode\hbox to3em{\hrulefill}\thinspace}
\providecommand{\MR}{\relax\ifhmode\unskip\space\fi MR }
\providecommand{\MRhref}[2]{%
  \href{http://www.ams.org/mathscinet-getitem?mr=#1}{#2}
}
\providecommand{\href}[2]{#2}

\end{document}